\documentclass[letterpaper, 10 pt, conference]{ieeeconf}
\IEEEoverridecommandlockouts

\usepackage[utf8]{inputenc} % allow utf-8 input
\usepackage[T1]{fontenc}    % use 8-bit T1 fonts
\usepackage{hyperref}       % hyperlinks
\usepackage{url}            % simple URL typesetting
\usepackage{booktabs}       % professional-quality tables
\usepackage{amsfonts}       % blackboard math symbols
\usepackage{nicefrac}       % compact symbols for 1/2, etc.
\usepackage{microtype}      % microtypography

\usepackage{cite}
\usepackage{amsmath}
\usepackage{amssymb}
\usepackage{amsfonts}
\usepackage{graphicx}
\usepackage{subfig}
\usepackage{textcomp}
\usepackage{xcolor}
\usepackage{epstopdf}
\usepackage{algorithm}
\usepackage{algorithmic}
\usepackage{balance}

\newtheorem{Thm}{Theorem}

\newtheorem{Lem}{Lemma}
\newtheorem{Ass}{Assumption}
\newtheorem{Rem}{Remark}
\newtheorem{Cor}{Corollary}

\title{\LARGE \bf  
On the Computation-Communication Trade-Off with A Flexible\\ Gradient Tracking Approach
}

\author{Yan Huang$^\dagger$ and Jinming Xu$^\dagger$ % <-this % stops a space
\thanks{This manuscript was submitted to the 62nd IEEE CDC in March 2023.}%
\thanks{$^\dagger$Yan Huang and Jinming Xu are with College of Control Science and Engineering, Zhejiang University, Hangzhou, China. E-mails: {\tt\small \{huangyan5616, jimmyxu\}@zju.edu.cn} }%
}%

\begin{document}

\maketitle

\begin{abstract}
We propose a flexible gradient tracking  approach with adjustable computation and communication steps for solving distributed stochastic optimization problem over networks. The proposed method allows each node to perform multiple local gradient updates and multiple inter-node communications in each round, aiming to strike a balance between computation and communication costs according to the properties of objective functions and network topology in non-i.i.d. settings. 
Leveraging a properly designed Lyapunov function, we derive both the computation and communication complexities for achieving arbitrary accuracy on smooth and strongly convex objective functions. Our analysis demonstrates sharp dependence of the convergence performance on graph topology and properties of objective functions, highlighting the trade-off between computation and communication. Numerical experiments are conducted to validate our theoretical findings.
\end{abstract}

\section{Introduction}

With the proliferation of individual computing devices and local collected user data \cite{konevcny2016federated}, distributed optimization methods have become increasingly popular in recent years due to their wide applications in various fields such as cooperative control \cite{nedic2018distributed}, distributed sensing \cite{akyildiz2011cooperative}, large-scale machine learning \cite{tsianos2012consensus}, and just to name a few. In this paper, we consider the standard distributed stochastic optimization problem jointly solved by a number of $n$ nodes over a network:
\begin{equation}\label{Prob} 
\underset{x\in \mathbb{R} ^p}{\min}f\left( x \right) =\frac{1}{n}\sum_{i=1}^n{\underset{:=f_i\left( x \right)}{\underbrace{\mathbb{E} _{\xi _i\sim \mathcal{D} _i}\left[ f_i\left( x;\xi _i \right) \right] }}},
\end{equation}
where $x\in \mathbb{R} ^p$ is the global decision variable and the objective function $f: \mathbb{R} ^p\rightarrow \mathbb{R}$ is the sum-utility of $n$ local objective function $f_i$ conditioned on the local data sample $\xi _i$ with distribution $\mathcal{D} _i$. As a promising approach, parallel/decentralized stochastic gradient decent \cite{mcmahan2017communication, lian2017can} is shown to be a simple yet efficient algorithm for solving the above distributed stochastic optimization problem \eqref{Prob} under certain scenarios. However, parallel/decentralized SGD may not ensure good performance in the presence of node-specific heterogeneity arising from {imbalanced} data sets, communication and computing resources \cite{wang2021field}.

% efficiency： local-update, gossip, compression, asynchronous
To avoid high communication burden among nodes, plenty of communication-efficient methods have been studied in the optimization and machine learning community. 
Particularly, Federated Averaging (FedAvg) \cite{mcmahan2017communication} (a.k.a. Local SGD \cite{stich2018local}), as a variant of parallel SGD with parameter server (PS) architecture \cite{dean2012large}, has been widely used in federated learning, which executes multiple local updates between two consecutive communication steps with partial or full node participation to save communication cost. The effectiveness of FedAvg/Local SGD for independent and identically distributed (i.i.d.) datasets has been extensively studied in the literature \cite{stich2018local, Lin2020Don't, woodworth2020local, koloskova2020unified}. For instance, it has been shown in \cite{Lin2020Don't, woodworth2020local} that Local SGD can outperform centralized mini-batch SGD \cite{bottou2010large} for quadratic objectives and certain convex cases. We note that these methods with PS architecture all require a central server for data aggregation, which may suffer from single-point failure and communication bottleneck \cite{lian2017can}. To address this issue, many decentralized SGD methods have been proposed for solving Problem \eqref{Prob} over peer-to-peer networks \cite{ram2009asynchronous, lian2017can}. In general, gossip-based communication protocols \cite{boyd2006randomized, lu2011gossip} are popular choices for distributed algorithms. For example, Lian \textit{et al.} \cite{lian2017can} proposed D-PSGD where each node communicates only with its neighbouring nodes for reaching consensus on optimization process, and, in \cite{lian2018asynchronous, ying2021exponential}, only {a subset} of the nodes are activated in each round {for exchanging information}, thus reducing communication costs.

While communication cost is a key concern in distributed optimization, it is equally important to ensure that the accuracy of the algorithm is not significantly compromised in practical scenarios. In particular, when it comes to non-i.i.d. settings where data distribution of nodes are heterogeneous, the adoption of local updates, partial participation and gossip protocols in parallel/decentralized SGD methods will introduce more degrees of data heterogeneity yielding poor algorithmic performance \cite{Li2020On}, and thus many variants have been proposed to address this issue. For instance, gradient estimation techniques \cite{ karimireddy2020scaffold, pu2021distributed, huang2022tackling, nguyen2022performance, huang2022stochastic, liu2023decentralized} and primal-dual-like methods \cite{yuan2018exact, yuan2020influence} have been shown to be effective for tackling data heterogeneity among nodes. In particular, Pu \textit{et al.} \cite{pu2021distributed} proposed a distributed stochastic gradient tracking (DSGT) method by introducing an auxiliary variable for each node to track the average gradient of local functions. To further reduce the communication cost, Nguyen \textit{et al.} \cite{nguyen2022performance} proposed a variant of DSGT, termed LU-GT, employing multiple local updates, and they provided the communication complexity matching Local SGD for non-convex objective functions. Building on this, Liu \textit{{et al.}} \cite{liu2023decentralized} proposed another variant adopting gradient-sum-tracking that further reduces the communication complexity with reduced stochastic gradient variance.
However, both the results in \cite{nguyen2022performance, liu2023decentralized} ignore the side effect of local updates that will amplify stochastic gradient noise on computation complexity.
The readers are referred to the recent survey papers \cite{tang2020communication, cao2023communication} for many other efforts to improve communication efficiency.

However, most existing distributed algorithms mainly focus on reducing communication costs without taking into account the acceleration of computation processes. To account for both computation and communication complexity, Berahas \textit{et al.} \cite{berahas2018balancing} proposed a variant of deterministic gradient decent method with multiple communication steps at each iteration, named NEAR-DGD, and they evaluated the performance of the algorithm via a new metric accounting for both communication and computation complexity, and showed that employing multiple consensus steps is desirable when communication is relatively cheap. Building on this, a stochastic variant called S-NEAR-DGD is proposed in \cite{iakovidou2022s} to accelerate the computation process. More recently, Liu \textit{et al.} \cite{liu2022decentralized} proposed a decentralized federated learning algorithm, named DFL, that employs multiple local updates and inter-node communication during each round and analyzed the impact of communication and computation on the performance of the algorithm separately.
Although these aforementioned algorithms are theoretically guaranteed for certain scenarios, their applications are limited by the assumption of uniformly bounded gradients. Moreover, they tend to be effective only with i.i.d. datasets and may face challenges with data heterogeneity in non-i.i.d. settings.

In this paper, we propose and analyze a flexible gradient tracking approach that employs adjustable communication and computation protocols for solving problem \eqref{Prob} in non-i.i.d. scenarios. 
The main contributions are summarized as: \textbf{i)} We develop a flexible gradient tracking method (termed FlexGT) employing multiple local updates and multiple inter-node communication, which enables it to deal with data heterogeneity and design customized protocols to balance communication and computation costs according to the properties of objective functions and network topology; \textbf{ii)} Our proposed algorithm comes with theoretical guarantees, including both communication and computation complexity analysis for strongly convex and smooth objective functions. In particular, we show that the proposed FlexGT algorithm  converges linearly to a neighborhood of the optimal solution regardless of data heterogeneity, which recovers the best known result of DSGT \cite{koloskova2021improved} as a special case under our settings. 
More importantly, the complexity results shed lights on adjusting the communication and computation frequency according to the problem characteristics to achieve a better trade-off. This is in contrast to the existing works \cite{iakovidou2022s, liu2022decentralized} that merely focus on communication or computation.

\textbf{Notations.} Throughout this paper, we adopt the following notations: $\left| \cdot \right|$ represents the Frobenius norm, $\mathbb{E}\left[ \cdot \right]$ denotes the expectation of a matrix or vector, $\mathbf{1}$ represents the all-ones vector, $\mathbf{I}$ denotes the identity matrix, and $\mathbf{J}=\mathbf{1}\mathbf{1}^T/n$ denotes the averaging matrix. 

\section{Problem Formulation and Algorithm Design}
We consider solving the distributed stochastic optimization problem \eqref{Prob} where $n$ agents are connected over a graph $\mathcal{G}=(\mathcal{V},\mathcal{E})$, Here, $\mathcal{V}=\{1,2,...,n\}$ represents the set of agents, and $\mathcal{E}\subseteq\mathcal{V}\times \mathcal{V}$ denotes the set of edges consisting of ordered pairs $(i,j)$ representing the communication link from node $j$ to node $i$. For node $i$, we define $\mathcal{N}_i= \{ j,|, \left( i, j \right) \in \mathcal{E} \}$ as the set of its neighboring nodes. We then make the following blanket assumptions on the objective function and graph.

\begin{Ass}\label{Ass_joint_smo_cov}
({\textbf{Convexity and smoothness}})
Each $f_i\left( x\right) $ is $\mu$-strongly convex and $L$-smooth in $x$.
\end{Ass}

Then, we assume each agent $i$ obtains an unbiased noisy gradient of the form $\nabla f_i\left( x; \xi _i \right) $ by querying a stochastic oracle ($\mathcal{S}\mathcal{O}$), which satisfies the following assumption.

\begin{Ass}\label{Ass_bounded_var}
(\textbf{Bounded variance})
For $\forall x, x^\prime \in \mathbb{R}^p$, there exist $\sigma  \geqslant 0$ such that
\begin{equation*}
\begin{aligned}
\mathbb{E} \left[ \left\| \nabla f_i\left( x;\xi _i \right) -\nabla f_i\left( x \right) \right\| ^2 \right] \leqslant \sigma ^2,
\end{aligned}
\end{equation*}
where the random sample $\xi_{i}$ is generated from $\mathcal{S}\mathcal{O}$.
\end{Ass}

\begin{Ass} \label{Ass_graph}
(\textbf{Graph connectivity})
The weight matrix $W$ induced by graph $\mathcal{G}$ is doubly stochastic, i.e., $W\mathbf{1}=\mathbf{1},\mathbf{1}^TW=\mathbf{1}^T$ and $\rho _W:= \left\| W-\mathbf{J} \right\|_2 ^2 <1.$

\end{Ass}

% \section{Personalized DSGT algorithm}
Now, we proceed to present the proposed FlexGT algorithm for solving problem (\ref{Prob}), which is given in Algorithm~\ref{Alg_Flexible_DSGT}.

\begin{algorithm}
\caption{\textbf{FlexGT}}
\begin{algorithmic}[1]\label{Alg_Flexible_DSGT}
\renewcommand{\algorithmicrequire}{\textbf{Initialization: }}
\REQUIRE {Initial points $x_{i,0}\in \mathbb{R}^p$ and $y_{i,0}=\nabla_x f_i\left( x_{i,0};\xi _{i,0} \right)$, communication and computation frequency $d_1, d_2 \geq 1$ and stepsize $\gamma >0$.
% defined in (\ref{Rules_step_size}).
}
\
\FOR{round $k = 0,1,\cdots$, each node $i\in[n]$,}

% \renewcommand{\algorithmicrequire}{\textbf{Local Variable Update:}}

% \renewcommand{\algorithmicrequire}{\textbf{Global Variable Update:}}
% \STATE \textbf{Local Computation:}
\FOR{$l=0,1,\cdots,d_2-1$}
\STATE {Obtain an unbiased noisy gradient sample of $\xi_{i,d_2k+l+1}$ by querying the stochastic oracle $\mathcal{S}\mathcal{O}$}.

\STATE Perform local update:
\begin{equation*}
\begin{aligned}
x_{i,d_2k+l+1}&=x_{i,d_2k+l}-\gamma y_{i,d_2k+l},
\\
y_{i,d_2k+l+1}&=\nabla f_i\left( x_{i,d_2k+l+1}; \xi _{i,d_2k+l+1} \right) 
\\
&+y_{i,d_2k+l}-\nabla f_i\left( x_{i,d_2k+l}; \xi _{i,d_2k+l} \right) .
\end{aligned}
\end{equation*}
\ENDFOR
\FOR{$s = 0,1,\cdots,d_1-1$}
\STATE Perform inter-node communication:
\begin{equation*}
\begin{aligned}
x_{i,d_2\left( k+1 \right)}&=\sum_{j\in \mathcal{N} _i}{W_{i,j}x_{j,d_2\left( k+1 \right)}},
\\
y_{i,d_2\left( k+1 \right)}&=\sum_{j\in \mathcal{N} _i}{W_{i,j}y_{j,d_2\left( k+1 \right)}}.
\end{aligned}
\end{equation*}
\ENDFOR

\ENDFOR
\end{algorithmic}
\end{algorithm}

For simplicity, we introduce the following notations:
\begin{equation*}
\begin{aligned}
X_t&:=\left[ x_{1,t},x_{2,t},\cdots ,x_{n,t} \right] ^T\in \mathbb{R}^{n\times p},
\\
Y_t&:=\left[ y_{1,t},y_{2,t},\cdots ,y_{n,t} \right] ^T\in \mathbb{R}^{n\times p},
\\
\nabla G_t&:=\left[ \cdots ,\nabla f_i\left( x_{i,t};\xi _{i,t} \right) ,\cdots \right] ^T\in \mathbb{R} ^{n\times p},
\\
\nabla F_t&:=\mathbb{E} \left[ \nabla G_t \right] =\left[ \cdots ,\nabla f_i\left( x_{i,t} \right) ,\cdots \right] ^T\in \mathbb{R} ^{n\times p}.
\end{aligned}
\end{equation*}
Then, Algorithm~\ref{Alg_Flexible_DSGT} can be rewritten in a compact form:
\begin{subequations}\label{Compact_PerDSGT}
\begin{align}
X_{d_2\left( k+1 \right)}&=W^{d_1}\left( X_{d_2k}-\gamma \sum_{j=0}^{d_2-1}{Y_{d_2k+j}} \right) ,\label{Eq_x_update}
\\
Y_{d_2\left( k+1 \right)}&=W^{d_1}\left( Y_{d_2k}+\nabla G_{d_2\left( k+1 \right)}-\nabla G_{d_2k} \right),\label{Eq_y_update}
\end{align}
\end{subequations}
where the integers $d_1, d_2 \in \left[ 1,\infty \right) $ denote the communication and computation frequency in each round respectively.

\begin{Rem}
% [\hy{Algorithm design}]
The flexibility of the proposed FlexGT algorithm consists in the adjustable communication and computation protocol with respect to $d_1$ and $d_2$, which allows for customized multiple local updates and message exchange over the network in each round according to the properties of objective functions and network topology. This algorithm can also eliminate the effect of data heterogeneity among nodes thanks to the gradient tracking scheme. It should be also noted that FlexGT recovers the standard DSGT \cite{pu2021distributed} algorithm ($d_1=d_2=1$) and LU-GT \cite{nguyen2022performance} ($d_1=1, d_2\geqslant1$) as special cases, and has the potential to recover S-NEAR-DGD \cite{iakovidou2022s} and DFL \cite{liu2022decentralized} as well if independent weight matrix for the update of $Y$ is employed. 
\end{Rem}

\section{Main Results}

In this section, we present the main convergence results of
the proposed FlexGT algorithm for strongly-convex and smooth objective functions. To this end, we first define the following Lyapunov function consisting of optimality gap, consensus error and gradient tracking error:
\begin{equation}\label{Def_Lyapunov_func}
\begin{aligned}
V_t:=\left\| \bar{x}_t-x^* \right\| ^2+c_1\left\| X_t-\mathbf{1}\bar{x}_t \right\| ^2+c_2\left\| Y_t-\mathbf{1}\bar{y}_t \right\| ^2,
\end{aligned}
\end{equation}
where $t=d_2k$, $c_1$ and $c_2$ are coefficients to be properly designed later, and
\begin{equation*}
\begin{aligned}
\bar{x}_t:=\frac{\mathbf{1}^T}{n}X_t,\quad \bar{y}_t:=\frac{\mathbf{1}^T}{n}Y_t.
\end{aligned}
\end{equation*}

With these definitions, we are ready to present the main convergence results of the proposed FlexGT algorithm.

\begin{Thm}\label{Thm_periodical}
Suppose the Assumptions \ref{Ass_joint_smo_cov}, \ref{Ass_bounded_var} and \ref{Ass_graph} hold. Let the stepsize satisfy
\begin{equation}\label{Eq_stepsize_1}
\begin{aligned}
\gamma \leqslant \min \left\{ \frac{1}{10d_2L},\frac{1-\rho _{W}^{d_1}}{37d_2L\sqrt[4]{\rho _{W}^{d_1}}}, \frac{\left( 1-\rho _{W}^{d_1} \right) ^2}{153d_2L\sqrt{\rho _{W}^{d_1}}} \right\}.
\end{aligned}
\end{equation}
Then, we have for all $k\geqslant0$,
\begin{equation}\label{Eq_Thm_stochastic}
\begin{aligned}
\mathbb{E} \left[ V_{d_2\left( k+1 \right)} \right] 
&\leqslant \left( 1-\min \left\{ \frac{\mu d_2 \gamma}{4}, \frac{1-\rho _{W}^{d_1}}{8} \right\} \right) \mathbb{E} \left[ V_{d_2k} \right] 
\\
&+\frac{d_2\gamma ^2\sigma ^2}{n}+\frac{d_{2}^{3}\gamma ^3L}{\left( 1-\rho _{W}^{d_1} \right) ^3}M_{\sigma},
\end{aligned}
\end{equation}
where $M_{\sigma}$ can be found in \eqref{Def_M_sigma}.
\end{Thm}

\begin{proof}
    See Section \ref{Proof_Of_Thm_1}.
\end{proof}

\begin{Cor}\label{Col_complexity}
Under the same setting of Theorem \ref{Thm_periodical}, the number of computation steps needed to achieve an accuracy of arbitrary small $\varepsilon \geqslant 0$ is
% for reaching arbitrary small $\varepsilon\geq 0$ accuracy, the computation complexity of FlexGT is,
\begin{equation}\label{Eq_comp_complexity}
\begin{aligned}
\tilde{\mathcal{O}}\left( \frac{d_2L}{\left( 1-\rho _{W}^{d_1} \right) ^2\mu}+\frac{\sigma ^2}{\mu ^2 n \varepsilon}+\frac{d_2\sqrt{L\rho _{W}^{d_1}\sigma ^2}}{\sqrt{\mu ^3\left( 1-\rho _{W}^{d_1} \right) ^3\varepsilon}} \right) ,
\end{aligned}
\end{equation}
and the number of communication steps needed is
\begin{equation}\label{Eq_comm_complexity}
\begin{aligned}
\tilde{\mathcal{O}}\left( \frac{d_1L}{\left( 1-\rho _{W}^{d_1} \right) ^2\mu}+\frac{d_1\sigma ^2}{\mu ^2 n  {d_2} \varepsilon}+\frac{d_1\sqrt{L\rho _{W}^{d_1}\sigma ^2}}{\sqrt{\mu ^3\left( 1-\rho _{W}^{d_1} \right) ^3\varepsilon}} \right),
\end{aligned}
\end{equation}
where the notation $\tilde{\mathcal{O}}(\cdot)$ hides the logarithmic factors.
\end{Cor}

\begin{proof}
    We omit the proof due to the space limitation. The techniques used to adapt Theorem \ref{Thm_periodical} to the complexity results are common and can be found in \cite{stich2019unified, koloskova2020unified, huang2022tackling}.
\end{proof}

\begin{Rem}
Theorem \ref{Thm_periodical} shows that the proposed FlexGT algorithm converges linearly to a neighborhood of the optimal solution of Problem \eqref{Prob}, depending on the problem characteristics as well as the  frequency of communication and computation. Particularly, this result implies that 
increasing $d_2$ can result in a smaller stepsize and smaller steady-state error, while increasing $d_1$ can lead to a faster linear rate but also a significant increase in communication overhead.
To further illustrate the insight, we derive the computation and communication complexity of FlexGT in Corollary \ref{Col_complexity}. Notably, this result demonstrates the existence of a trade-off between the computation and communication complexity, that is, as the computation ({resp.} communication) frequency $d_2$ ({resp.} $d_1$) increases, the computation complexity will increase ({resp.} decrease), while the communication complexity will decrease ({resp.} increase or decrease depending on $\rho_W$), thereby calling for a careful design of $d_1$ and $d_2$ to balance communication and computation costs based on the problem setting. Moreover, compared to the existing results in \cite{liu2022decentralized, iakovidou2022s}, the convergence of FlexGT is independent of data heterogeneity among nodes captured as $\zeta _f:=\mathbf{sup}\left\{ \frac{1}{n}\sum_{i=1}^n{\left\| \nabla f_i\left( x \right) -\nabla f\left( x \right) \right\| ^2} \right\} , \forall x$ \cite{lian2017can}, without the assumption of uniformly bounded stochastic gradient, and is thus more robust in non-i.i.d. settings.
\end{Rem}

\section{Convergence Analysis}
In this section, we carry out the convergence analysis for the main results. We begin by introducing the following key lemmas that are crucial for the proof of Theorem \ref{Thm_periodical}.

\subsection{Key Lemmas}

\begin{Lem}[\textbf{Bounding optimality gap}]\label{Lem_opt_gap}
Suppose Assumptions \ref{Ass_joint_smo_cov}, \ref{Ass_bounded_var} and \ref{Ass_graph} hold. Let the setp-size satisfy $\gamma \leqslant \frac{1}{10d_2L}$. Then, we have for $\forall k \geqslant 0$,
\begin{equation}
\begin{aligned}
&\mathbb{E} \left[ \left\| \bar{x}_{d_2\left( k+1 \right)}-x^* \right\| ^2 \right] 
\\
&\leqslant \left( 1-\frac{d_2\mu \gamma}{2} \right) \mathbb{E} \left[ \left\| \bar{x}_{d_2k}-x^* \right\| ^2 \right] 
\\
&+\frac{12d_2\gamma L}{n}\mathbb{E} \left[ \left\| X_{d_2k}-\mathbf{1}\bar{x}_{d_2k} \right\| ^2 \right] 
\\
&+\frac{12d_{2}^{3}\gamma ^3L}{n}\mathbb{E} \left[ \left\| Y_{d_2k}-\mathbf{1}\bar{y}_{d_2k} \right\| ^2 \right] 
\\
&-d_2\gamma \mathbb{E} \left[ f\left( \bar{x}_{d_2k} \right) -f\left( x^* \right) \right] +\frac{d_2\gamma ^2\sigma ^2}{n}+36d_{2}^{3}\gamma ^3L\sigma ^2.
\end{aligned}
\end{equation}
\end{Lem}

\begin{proof}
    See Appendix \ref{Proof_of_lem_opt_gap}.
\end{proof}

\begin{Lem}[\textbf{Bounding consensus error}]\label{Lem_cons_err}
Suppose Assumptions \ref{Ass_joint_smo_cov}, \ref{Ass_bounded_var} and \ref{Ass_graph} hold. Let the stepsize satisfy
\begin{equation}\label{Eq_step_cons_err}
\gamma \leqslant \min \left\{ \frac{1}{8d_2L}, \frac{1-\rho _{W}^{d_1} }{12d_2L\sqrt{ \rho _{W}^{d_1}}} \right\}.
\end{equation}
Then, we have for $\forall k \geqslant 0$, 
\begin{equation}
\begin{aligned}
&\mathbb{E} \left[ \left\| X_{d_2\left( k+1 \right)}-\mathbf{1}\bar{x}_{d_2\left( k+1 \right)} \right\| ^2 \right] 
\\
&\leqslant \frac{3+\rho _{W}^{d_1}}{4}\mathbb{E} \left[ \left\| X_{d_2k}-\mathbf{1}\bar{x}_{d_2k} \right\| ^2 \right] 
\\
&+\frac{192nd_{2}^{4}\gamma ^4L^3\rho _{W}^{d_1}}{1-\rho _{W}^{d_1}}\mathbb{E} \left[ f\left( \bar{x}_{d_2k} \right) -f\left( x^* \right) \right] 
\\
&+\frac{6d_{2}^{2}\gamma ^2\rho _{W}^{d_1}}{1-\rho _{W}^{d_1}}\mathbb{E} \left[ \left\| Y_{d_2k}-\mathbf{1}\bar{y}_{d_2k} \right\| ^2 \right] +\frac{18nd_{2}^{2}\gamma ^2\rho _{W}^{d_1}}{1-\rho _{W}^{d_1}}\sigma ^2.
\end{aligned}
\end{equation}
\end{Lem}

\begin{proof}
    See Appendix \ref{Proof_of_lem_con_err}.
\end{proof}

\begin{Lem}[\textbf{Bounding tracking error}]\label{Lem_tra_err}
Suppose Assumptions \ref{Ass_joint_smo_cov}, \ref{Ass_bounded_var} and \ref{Ass_graph} hold. Let the stepsize satisfy
\begin{equation}\label{Eq_step_tracking_err}
\gamma \leqslant \left\{ \frac{1}{8d_2L}, \frac{1-\rho _{W}^{d_1}}{10d_2L\sqrt{\rho _{W}^{d_1}}} \right\}.
\end{equation}
Then, we have for $\forall k \geqslant 0$,
\begin{equation}
\begin{aligned}
&\mathbb{E} \left[ \left\| Y_{d_2\left( k+1 \right)}-\mathbf{1}\bar{y}_{d_2\left( k+1 \right)} \right\| ^2 \right] 
\\
&\leqslant \frac{3+\rho _{W}^{d_1}}{4}\mathbb{E} \left[ \left\| Y_{d_2k}-\mathbf{1}\bar{y}_{d_2k} \right\| ^2 \right] 
\\
&+\frac{30\rho _{W}^{d_1}L^2}{1-\rho _{W}^{d_1}}\mathbb{E} \left[ \left\| X_{d_2k}-\mathbf{1}\bar{x}_{d_2k} \right\| ^2 \right] 
\\
&+\frac{96n\rho _{W}^{d_1}d_{2}^{2}\gamma ^2L^3}{1-\rho _{W}^{d_1}}\mathbb{E} \left[ f\left( \bar{x}_{d_2k} \right) -f\left( x^* \right) \right] +6n\rho _{W}^{d_1}\sigma ^2.
\end{aligned}
\end{equation}
\end{Lem}

\begin{proof}
    See Appendix \ref{Proof_of_lem_tra_err}.
\end{proof}

\subsection{Proofs of Main Results}

\subsubsection{Proof of Theorem \ref{Thm_periodical}}\label{Proof_Of_Thm_1}
Recalling the defined Lyapunov function \eqref{Def_Lyapunov_func} with
\begin{equation}\label{Eq_c1_c2}
\begin{aligned}
c_1=\frac{192d_2\gamma L}{n\left( 1-\rho _{W}^{d_1} \right)},\quad c_2=\frac{9312d_{2}^{3}\gamma ^3L}{n\left( 1-\rho _{W}^{d_1} \right) ^3}.
\end{aligned}
\end{equation}
Then, with the help of Lemma \ref{Lem_opt_gap}, \ref{Lem_cons_err} and \ref{Lem_tra_err}, we can bound the Lyapunov function as follows:
\begin{equation}
\begin{aligned}
&\mathbb{E} \left[ V_{d_2\left( k+1 \right)} \right] 
\leqslant \left( 1-\min \left\{ \frac{\gamma \mu d_2}{4},\frac{1-\rho _{W}^{d_1}}{8} \right\} \right) \mathbb{E} \left[ V_{d_2k} \right] 
\\
&\leqslant \frac{d_2\gamma ^2\sigma ^2}{n}+\frac{d_{2}^{3}\gamma ^3L}{\left( 1-\rho _{W}^{d_1} \right) ^3}M_{\sigma}+e_1\mathbb{E} \left[ f\left( \bar{x}_{d_2k} \right) -f\left( x^* \right) \right] 
\\
&+e_2\mathbb{E} \left[ \left\| X_{d_2k}-\mathbf{1}\bar{x}_{d_2k} \right\| ^2 \right] +e_3\mathbb{E} \left[ \left\| Y_{d_2k}-\mathbf{1}\bar{y}_k \right\| ^2 \right] ,
\end{aligned}
\end{equation}
where
\begin{equation*}
\begin{aligned}
e_1&:=c_1\frac{192nd_{2}^{4}\gamma ^4L^3\rho _{W}^{d_1}}{1-\rho _{W}^{d_1}}+c_2\frac{96n\rho _{W}^{d_1}d_{2}^{2}\gamma ^2L^3}{1-\rho _{W}^{d_1}}-d_2\gamma ,
\\
e_2&:=\frac{12d_2\gamma L}{n}+c_2\frac{30\rho _{W}^{d_1}L^2}{1-\rho _{W}^{d_1}}-c_1\frac{1-\rho _{W}^{d_1}}{8},
\\
e_3&:=\frac{12d_{2}^{3}\gamma ^3L}{n}+c_1\frac{6d_{2}^{2}\gamma ^2\rho _{W}^{d_1}}{1-\rho _{W}^{d_1}}-c_2\frac{1-\rho _{W}^{d_1}}{8},
\end{aligned}
\end{equation*}
and
\begin{equation}\label{Def_M_sigma}
\begin{aligned}
M_{\sigma}&:=36\left( 1-\rho _{W}^{d_1} \right) ^3\sigma ^2+3456\left( 1-\rho _{W}^{d_1} \right) \rho _{W}^{d_1}\sigma ^2
\\
&+55872\rho _{W}^{d_1}\sigma ^2.
\end{aligned}
\end{equation}
Letting $e_1, e_2, e_3\leqslant 0$ with the stepsize $\gamma$ satisfying \eqref{Eq_stepsize_1}, we obtain
the result in \eqref{Eq_Thm_stochastic}, which completes the proof.

\begin{Rem}
% [\hy{Technical contribution}]
The proof of the main results relies on double-loop analysis and a properly designed Lyapunov function as depicted in \eqref{Def_Lyapunov_func}. Specifically, we first derive the upper bounds of consensus error and gradient tracking error within a period (inner loop), respectively (c.f., Lemmas \ref{Lem_inner_cons_err} and \ref{Lem_inner_tra_err}). These bounds are then used to establish the contraction properties of the consensus error and gradient tracking error at each round (outer loop) (c.f., Lemmas \ref{Lem_cons_err} and \ref{Lem_tra_err}).  Finally, we combine these obtained error terms to construct the Lyapunov function with properly designed coefficients \eqref{Eq_c1_c2}, which allows us to obtain a rate result that shows the sharp dependence of convergence performance on the properties of objective functions, network topology as well as the frequency of communication and computation. 
% we obtain the linear convergence result which reveals a clear dependency of linear rate and steady-state error on the properties of objective functions, network topology and the frequency of communication and computation. 
It should be noted that this result also matches the best-known rate of the DSGT \cite{koloskova2021improved} (a special case with $d_1=d_2=1$) under the same setting, i.e., $p=c=1-\rho _W$.
\end{Rem}

\section{Numerical Experiments}
In this section, we report a series of numerical experiments to verify the theoretical findings of the proposed FlexGT algorithm by means of a synthetic example. Specifically, we consider the following quadratic function:
\begin{equation}
\begin{aligned}
\underset{x}{\min}f\left( x \right) =\frac{1}{n}\sum_{i=1}^n{\underset{=:f_i}{\underbrace{\left( \mathbb{E} _{v_i}\left[ \left( h_{i}^{T}x-\nu _i \right) ^2+\frac{\mu}{2}\left\| x \right\| ^2 \right] \right) }}},
\end{aligned}
\end{equation}
where $\mu \geqslant0$ is the regularization parameter, $h_i\in \left[ 0,1 \right] ^p$ denotes the feature parameters of node $i$ with dimension $p=10$, and $v_i\sim \mathcal{N} \left( \bar{v}_i,\sigma ^2 \right)$ with $\bar{v}_i\in \left[ 0,1 \right]$. Therefore, the algorithm can obtain an unbiased noisy gradient $g_{i}\left( x_{i,t} \right) :=\nabla f_i\left( x_{i,t} \right) +\delta _{i,t}$ with $\delta _{i,t}\sim \mathcal{N} \left( 0,\sigma^2 \right) $ at each iteration $t$.
Moreover, we set the stepsize according to the choice of $d_1$ and $d_2$, i.e., $\gamma =c\left( 1-\rho _{W}^{d_1} \right) ^2/\left( d_2L \right) $, where $c=10$ is a constant and Lipschitz constant is set to $L=1$.

\textbf{Communication and computation trade-off.} 
To balance the communication and computation costs in practice, we attempt to minimize the weighted-sum of the obtained communication and computation complexity results:
\begin{equation}\label{Eq_weighted_sum_compl}
\begin{aligned}
\underset{d_1, d_2}{\min}\,\,\left( \omega _1\mathcal{C} _1+\omega _2\mathcal{C} _2 \right),
\end{aligned}
\end{equation}
where $\mathcal{C} _1$ and $\mathcal{C} _2$ represent the obtained communication and computation complexity respectively, $\omega _1$ and $\omega _2$ are the corresponding weights. We note that $\mathcal{C} _2=\mathcal{C} _1d_1/d_2$, and thus plot the heat-map of weighted complexity of FlexGT to achieve $\varepsilon = 10^{-5}$ accuracy with different $d_1$ and $d_2$ under specific settings in Fig.~ \ref{Fig_3}. It can be observed that the overal complexity of FlexGT varies with the increase of $d_1$ or $d_2$ and achieves the best performance at $d_1=3$ and $d_2=2$ (see green box on the left). Moreover, if we keep the radio of $d_2$ and $d_1$ fixed, i.e., $d_2/d_1=d$, as shown on the right, there exists an optimal ratio to minimizing the weighted-sum complexity \eqref{Eq_weighted_sum_compl}. These observations illustrate the trade-offs between communication and computation conditioned on the properties of objective functions and graph topology.

\textbf{Eliminating the effects of node heterogeneity.}
In Fig.~\ref{Fig_4}, we compare the convergence performance of the proposed FlexGT algorithm with DFL \cite{liu2022decentralized} and their special cases DSGT (FlexGT with $d_1=d_2=1$) and D-PSGD (DFL with $d_1=d_2=1$) in terms of computation and communication steps. We note that there is heterogeneity among the objectives $f_i$ of nodes due to differences in $\{ h_i \} _{i=1}^{n}$. In this scenario, it can be observed that FlexGT has a significant advantage over DFL in terms of steady-state error thanks to the gradient tracking method. Moreover, the choice of $d_1=3$ and $d_2=2$ makes FlexGT achieve better computation and communication complexity to reach an accuracy of $\varepsilon=10^{-5}$.

\begin{figure}[!tb]
    \centering
    \subfloat %子图
    {
        \begin{minipage}[t]{0.22\textwidth}
            \centering          %子图居中
            \includegraphics[width=\textwidth]{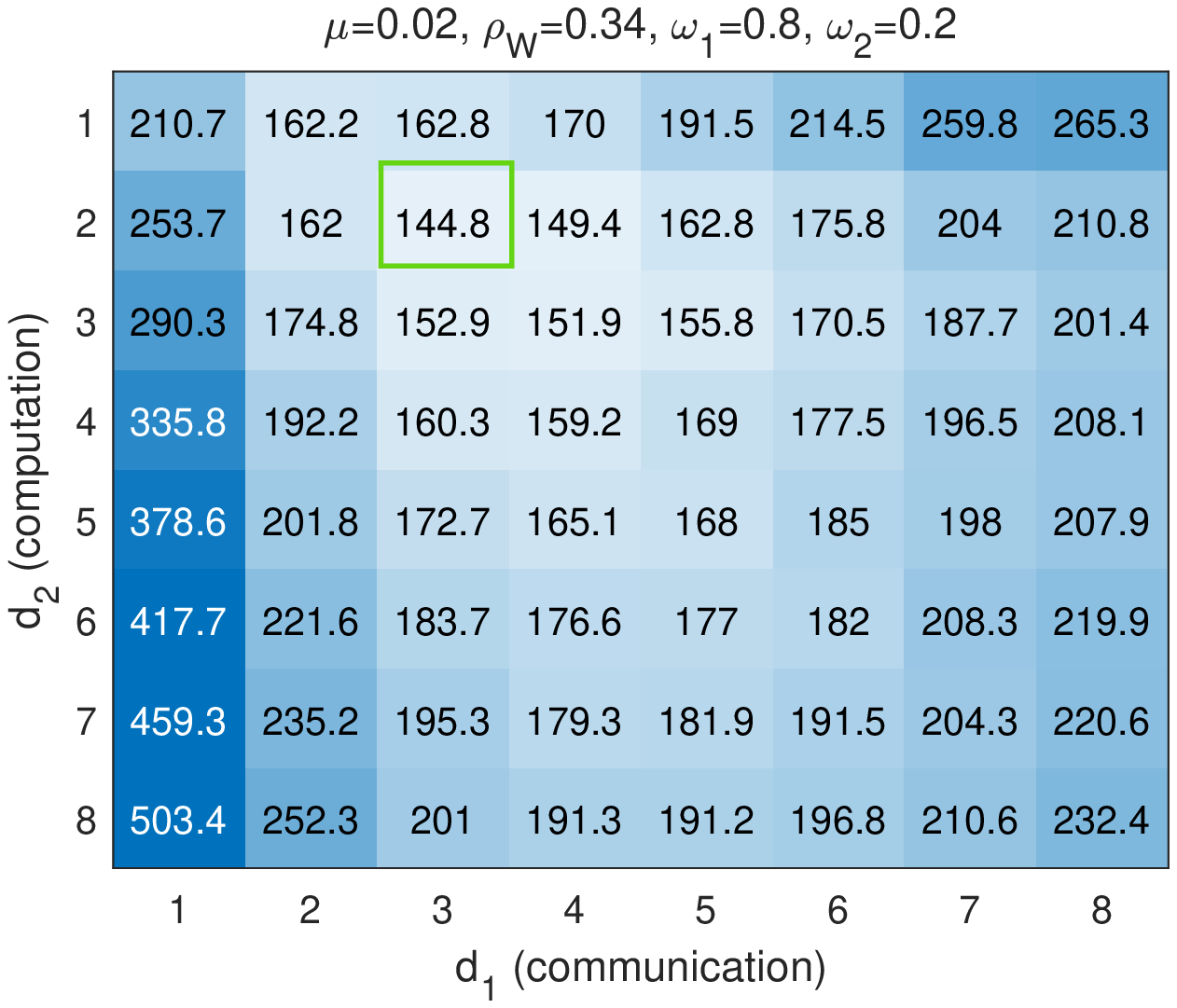}   %以行宽的0.5倍大小显示
        \end{minipage}%
    }
    \subfloat %子图
    {
        \begin{minipage}[t]{0.255\textwidth}
            \centering          %子图居中
            \includegraphics[width=\textwidth]{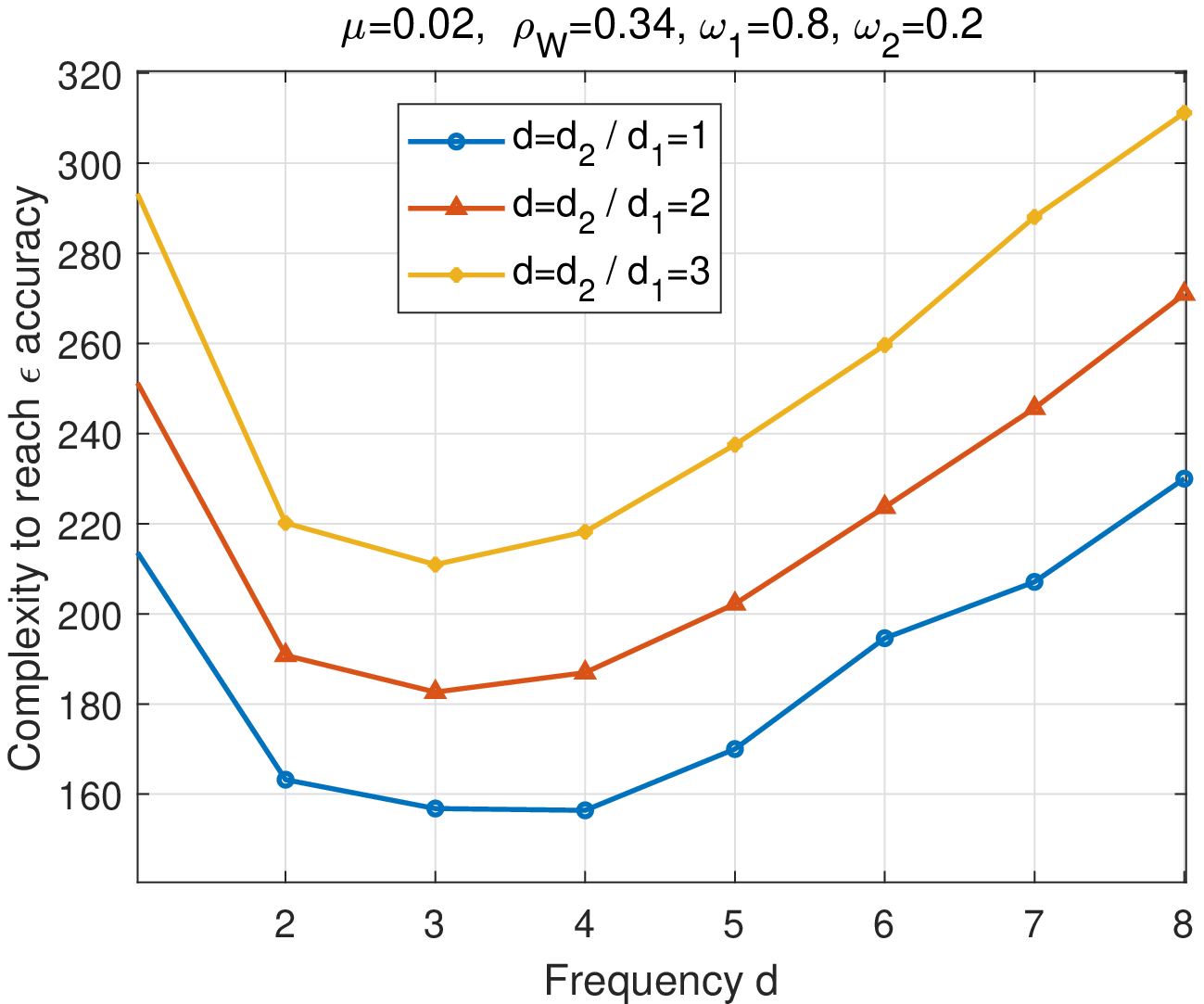}   %以行宽的0.5倍大小显示
        \end{minipage}%
    }   
    \caption{Weighted-sum complexity of FlexGT algorithm to achieve $\varepsilon=10^{-5}$ accuracy with different communication and computation frequency in each round over exponential graphs of $n=20$ nodes with $N_i=5$ neighbors.}
    \label{Fig_3}
    \vspace{-0.4cm}
\end{figure}

\begin{figure}[!tb]
    \centering
    \subfloat %子图
    {
        \begin{minipage}[t]{0.245\textwidth}
            \centering          %子图居中
            \includegraphics[width=\textwidth]{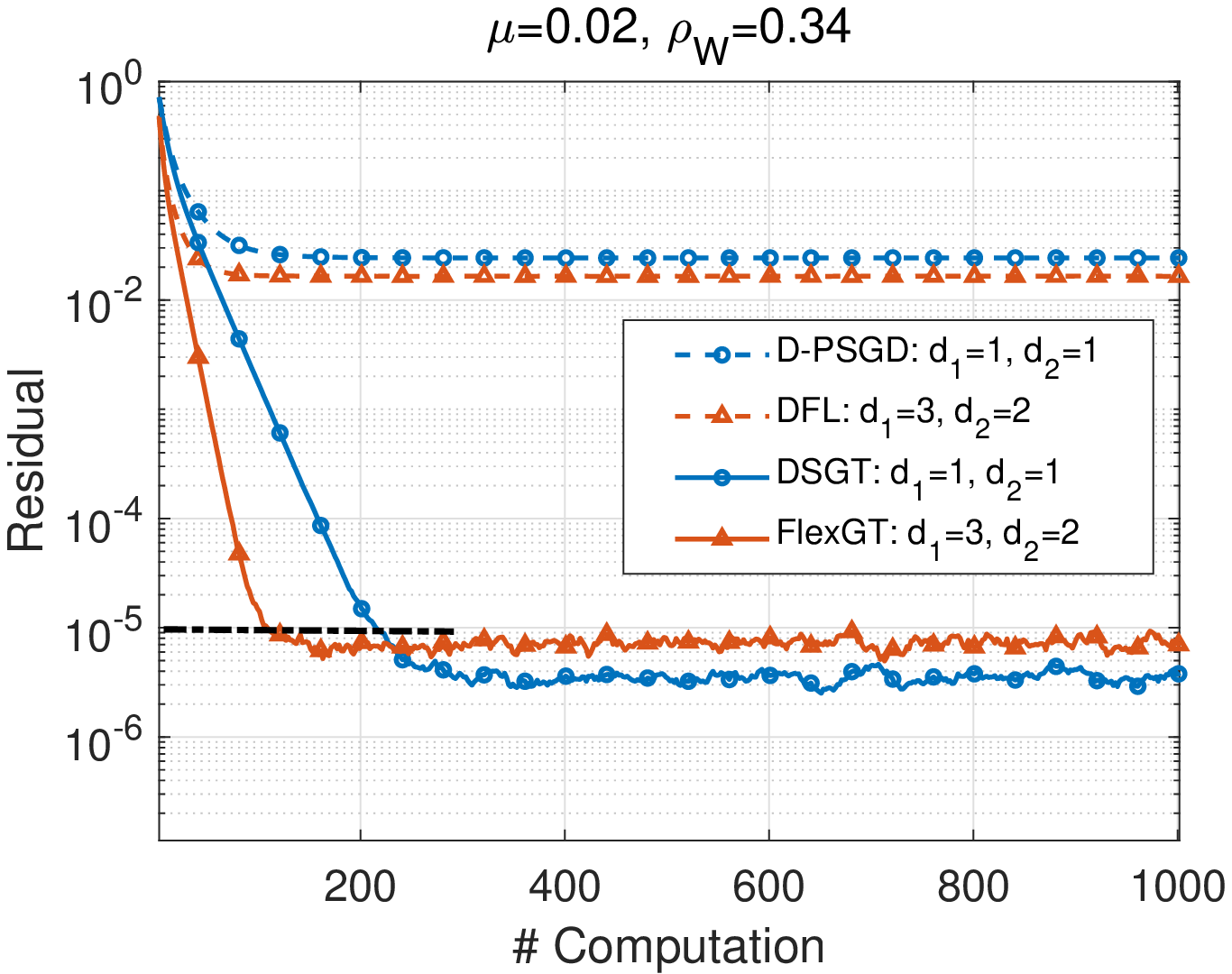}   %以行宽的0.5倍大小显示
        \end{minipage}%
    }
    \subfloat %子图
    {
        \begin{minipage}[t]{0.245\textwidth}
            \centering          %子图居中
            \includegraphics[width=\textwidth]{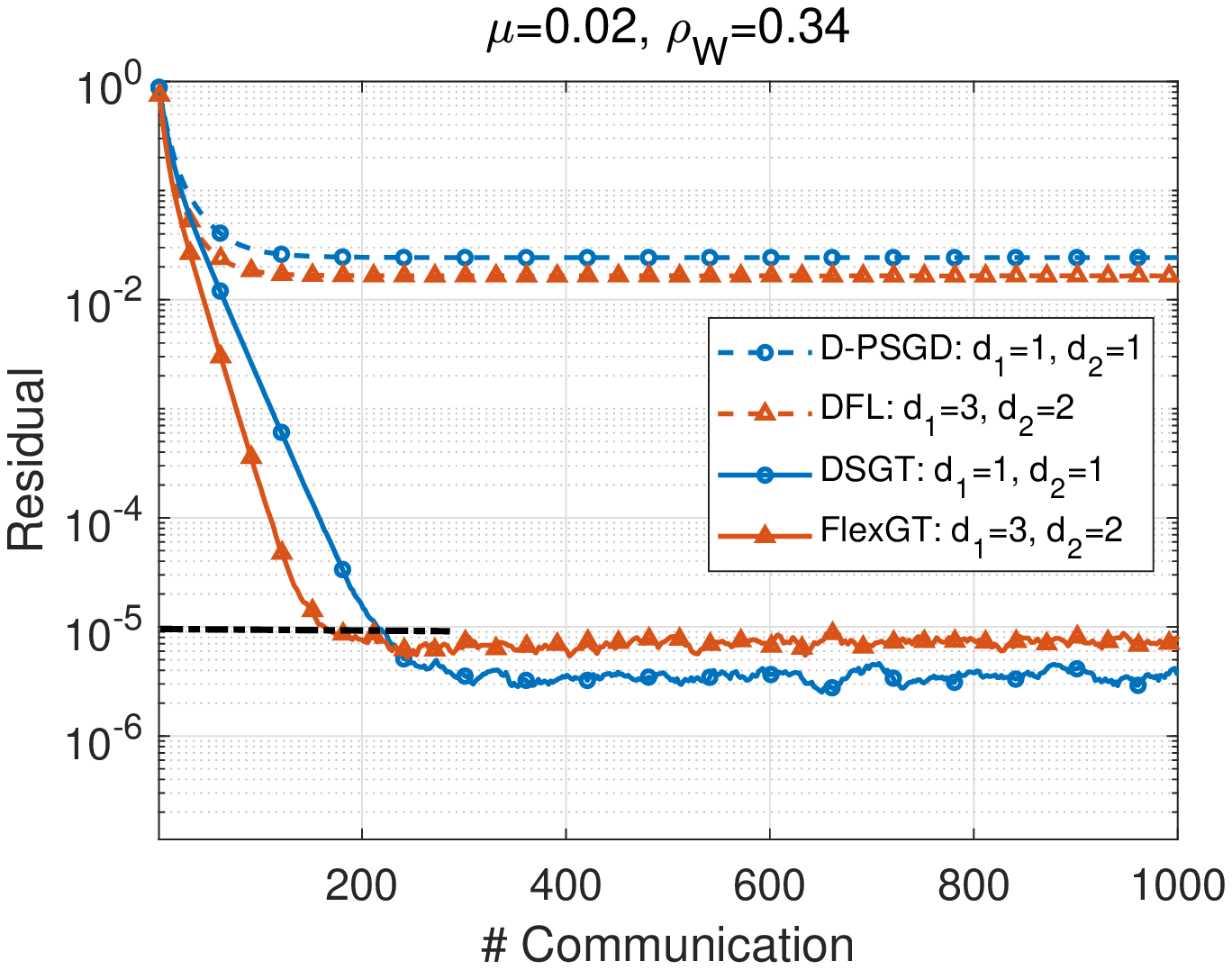}   %以行宽的0.5倍大小显示
        \end{minipage}%
    }   
    \caption{Comparison for FlexGT, DFL, standard D-PSGD and DSGT algorithms over exponential graphs of $n=20$ nodes with $N_i=5$ neighbors.}
    \label{Fig_4}
    \vspace{-0.5cm}
\end{figure}

\section{Conclusions}
In this paper, we proposed a flexible and efficient distributed stochastic gradient tracking method FlexGT for solving distributed stochastic optimization problems under non-i.i.d settings. Our approach is able to handle data heterogeneity and design adjustable protocols to balance communication and computation costs according to the properties of the objective functions and network topology. We also provided theoretical guarantees for the proposed algorithm, including communication and computation complexity analysis for strongly convex and smooth objective functions, regardless of heterogeneity among nodes. These results provided an intuitive way to tune communication and computation protocols, highlighting their trade-offs. It will also be important to extend these results into more general settings, such as non-convex and time-varying cases, in the future work.

\appendix
\section{Appendix}
% \onecolumn

In this section, we provide the missing proofs for the lemmas and theorem in the main text. To this end, we first provide several supporting lemmas for the analysis.

\subsection{Supporting Lemmas}\label{App_Lem_1}

\begin{Lem}[Bernoulli's inequality] \label{Lem_tech_1}
For constants $\beta \geqslant 1$ and $\lambda>0$, we have
\begin{equation}
\left( 1+\frac{\lambda}{\beta} \right) ^{\beta}\leqslant e^{\lambda}.
\end{equation}
\end{Lem}

\begin{Lem}[\textbf{Bounding client divergence}]\label{Lem_tech_client_diver}
Suppose Assumptions \ref{Ass_joint_smo_cov}, \ref{Ass_bounded_var} and \ref{Ass_graph} hold. Let the stepsize satisfy $\gamma \leqslant \frac{1}{8d_2L}$. Then, we have for $k\geqslant 0$ and integer $t\in \left[ 1,d_2-1 \right] $, $d_2\geqslant2$, 
\begin{equation}
\begin{aligned}
&\mathbb{E} \left[ \left\| X_{d_2k+t}-\mathbf{1}\bar{x}_{d_2k} \right\| ^2 \right] 
\\
&\leqslant 4\mathbb{E} \left[ \left\| X_{d_2k}-\mathbf{1}\bar{x}_{d_2k} \right\| ^2 \right] +4d_{2}^{2}\gamma ^2\mathbb{E} \left[ \left\| Y_{d_2k}-\mathbf{1}\bar{y}_k \right\| ^2 \right] 
\\
&+16nd_{2}^{2}\gamma ^2L\mathbb{E} \left[ f\left( \bar{x}_{d_2k} \right) -f\left( x^* \right) \right] +d_{2}^{2}\gamma ^2\left( 4\sigma ^2+8n\sigma ^2 \right).
\end{aligned}
\end{equation}
\end{Lem}
\begin{proof}
By the x-update of FlexGT \eqref{Eq_x_update}, noticing that $
Y_{d_2k+t}=Y_{d_2k}+\nabla G_{d_2k+t}-\nabla G_{d_2k}
$, we have
\begin{equation}
\begin{aligned}
&\mathbb{E} \left[ \left\| X_{d_2k+t}-\mathbf{1}\bar{x}_{d_2k} \right\| ^2 \right] 
\\
&\overset{\left( a \right)}{\leqslant}\left( 1+\beta \right) \mathbb{E} \left[ \left\| X_{d_2k+t-1}-\mathbf{1}\bar{x}_{d_2k} \right\| ^2 \right] 
\\
&+\left( 1+\frac{1}{\beta} \right) \gamma ^2\mathbb{E} \left[ \left\| Y_{d_2k}+\nabla G_{d_2k+t-1}-\nabla G_{d_2k} \right\| ^2 \right] 
\\
&\overset{\left( b \right)}{\leqslant}\left( 1+\beta \right) \mathbb{E} \left[ \left\| X_{d_2k+t-1}-\mathbf{1}\bar{x}_{d_2k} \right\| ^2 \right] +8\left( 1+\frac{1}{\beta} \right) \gamma ^2n\sigma ^2
\\
&+4\left( 1+\frac{1}{\beta} \right) \gamma ^2\left( \mathbb{E} \left[ \left\| Y_{d_2k}-\mathbf{1}\bar{y}_k \right\| ^2 \right] +\mathbb{E} \left[ \left\| \mathbf{1}\bar{y}_k \right\| ^2 \right] \right) 
\\
&+4\left( 1+\frac{1}{\beta} \right) \gamma ^2\mathbb{E} \left[ \left\| \nabla F_{d_2k+t-1}-\nabla F_{d_2k} \right\| ^2 \right] 
\\
&\overset{\left( c \right)}{\leqslant}\left( 1+\beta +4\left( 1+\frac{1}{\beta} \right) \gamma ^2L^2 \right) ^t\mathbb{E} \left[ \left\| X_{d_2k}-\mathbf{1}\bar{x}_{d_2k} \right\| ^2 \right] 
\\
&+4t\left( 1+\frac{1}{\beta} \right) \gamma ^2\mathbb{E} \left[ \left\| Y_{d_2k}-\mathbf{1}\bar{y}_k \right\| ^2 \right] 
\\
&+4t\left( 1+\frac{1}{\beta} \right) n\gamma ^2\mathbb{E} \left[ \left\| \bar{y}_k \right\| ^2 \right] +8t\left( 1+\frac{1}{\beta} \right) \gamma ^2n\sigma ^2,
\end{aligned}
\end{equation}
where we have used Young's inequality with parameter $\beta$ in $(a)$, the bounded stochastic variance of Assumption \ref{Ass_bounded_var} in $(b)$ and the smoothness of $f_i$ of Assumption \ref{Ass_joint_smo_cov} in $(c)$. Then, by Lemma \ref{Lem_tech_1} and letting $\beta =1/(d_2-1)$, $\gamma \leqslant 1/(8d_2L)$, we get
% \begin{equation*}
% \beta =\frac{}{d_2-1},\,\, \gamma \leqslant \frac{1}{8d_2L},
% \end{equation*}
% we get
\begin{equation}
\begin{aligned}
&\left( 1+\frac{1}{d_2-1}+4d_2L^2\gamma ^2 \right) ^{t}
\leqslant e^{1+1/16}<3.
\end{aligned}
\end{equation}
Noticing that 
\begin{equation}
\begin{aligned}
\mathbb{E} \left[ \left\| \bar{y}_k \right\| ^2 \right] &\leqslant \frac{2L^2}{n}\left[ \left\| X_{d_2k}-\mathbf{1}\bar{x}_{d_2k} \right\| ^2 \right] 
\\
&+4L\mathbb{E} \left[ f\left( \bar{x}_{d_2k} \right) -f\left( x^* \right) \right] +\frac{\sigma ^2}{n},
\end{aligned}
\end{equation}
and $t < d_2$, we complete the proof.
\end{proof}

\begin{Lem}[\textbf{Bounding tracking error within a period}]\label{Lem_inner_tra_err}
Suppose Assumptions \ref{Ass_joint_smo_cov}, \ref{Ass_bounded_var} and \ref{Ass_graph} hold. Let the stepsize satisfy $\gamma \leqslant \frac{1}{8d_2L}$. Then,
we have for $k\geqslant 0$ and integer $t\in \left[ 1,d_2-1 \right]$, $d_2\geqslant2$, 
\begin{equation}
\begin{aligned}
&\mathbb{E} \left[ \left\| Y_{d_2k+t}-\mathbf{1}\bar{y}_{d_2k+t} \right\| ^2 \right] 
\\
&\leqslant 3\mathbb{E} \left[ \left\| Y_{d_2k}-\mathbf{1}\bar{y}_{d_2k} \right\| ^2 \right] +16L^2\mathbb{E} \left[ \left\| X_{d_2k}-\mathbf{1}\bar{x}_{d_2k} \right\| ^2 \right] 
\\
&+96nd_{2}^{2}\gamma ^2L^3\mathbb{E} \left[ f\left( \bar{x}_{d_2k} \right) -f\left( x^* \right) \right] +9n\sigma ^2.
\end{aligned}
\end{equation}
\end{Lem}

\begin{proof}
By the y-update of FlexGT \eqref{Eq_y_update}, noticing that $
Y_{d_2k+t}=Y_{d_2k}+\nabla G_{d_2k+t}-\nabla G_{d_2k}
$, we have
\begin{equation}
\begin{aligned}
&\mathbb{E} \left[ \left\| Y_{d_2k+t}-\mathbf{1}\bar{y}_{d_2k+t} \right\| ^2 \right] 
\\
&\leqslant 2\mathbb{E} \left[ \left\| Y_{d_2k}-\mathbf{1}\bar{y}_{d_2k} \right\| ^2 \right] +4\mathbb{E} \left[ \left\| \nabla F_{d_2k+t}-\nabla F_{d_2k} \right\| ^2 \right] 
\\
&+4\mathbb{E} \left[ \left\| \nabla G_{d_2k+t}-\nabla F_{d_2k+t}+\nabla F_{d_2k}-\nabla G_{d_2k} \right\| ^2 \right] 
\\
&\leqslant 2\mathbb{E} \left[ \left\| Y_{d_2k}-\mathbf{1}\bar{y}_{d_2k} \right\| ^2 \right] 
\\
&+4L^2\underset{=:S_1}{\underbrace{\mathbb{E} \left[ \left\| X_{d_2k+t}-\mathbf{1}\bar{x}_{d_2k} \right\| ^2 \right] }}+8n\sigma ^2.
\end{aligned}
\end{equation}
Using Lemma \ref{Lem_tech_client_diver} to further bound $S_1$, 
and letting the stepsize $\gamma \leqslant \frac{1}{8d_2L}$, we complete the proof.
\end{proof}

\begin{Lem}[\textbf{Bounding consensus error within a period}]\label{Lem_inner_cons_err}
Suppose Assumptions \ref{Ass_joint_smo_cov}, \ref{Ass_bounded_var} and \ref{Ass_graph} hold. Let the stepsize satisfy $\gamma \leqslant \frac{1}{8d_2L}$. Then, 
we have for $k\geqslant 0$ and integer $t \in \left[ 1,d_2-1 \right] $, $d_2\geqslant2$, 
\begin{equation}
\begin{aligned}
&\mathbb{E} \left[ \left\| X_{d_2k+t}-\mathbf{1}\bar{x}_{d_2k+t} \right\| ^2 \right] 
\\
&\leqslant 3\mathbb{E} \left[ \left\| X_{d_2k}-\mathbf{1}\bar{x}_{d_2k} \right\| ^2 \right] +3d_{2}^{2}\gamma ^2\mathbb{E} \left[ \left\| Y_{d_2k}-\mathbf{1}\bar{y}_{d_2k} \right\| ^2 \right] 
\\
&+128nd_{2}^{4}\gamma ^4L^3\mathbb{E} \left[ f\left( \bar{x}_{d_2k} \right) -f\left( x^* \right) \right] +9d_{2}^{2}\gamma ^2n\sigma ^2.
\end{aligned}
\end{equation}
\end{Lem}

\begin{proof}
Using Young's inequality, we obtain
\begin{equation}
\begin{aligned}
&\mathbb{E} \left[ \left\| X_{d_2k+t}-\mathbf{1}\bar{x}_{d_2k+t} \right\| ^2 \right] 
\\
&\leqslant \left( 1+\frac{1}{d_2-1} \right) \mathbb{E} \left[ \left\| X_{d_2k+t-1}-\mathbf{1}\bar{x}_{d_2k+t-1} \right\| ^2 \right] 
\\
&+d_2\gamma ^2\mathbb{E} \left[ \left\| Y_{d_2k+t-1}-\mathbf{1}\bar{y}_{d_2k+t-1} \right\| ^2 \right] 
\\
&\overset{\left( a \right)}{\leqslant}\left( \left( 1+\frac{1}{d_2-1} \right) ^t+16d_{2}^{2}\gamma ^2L^2 \right) \mathbb{E} \left[ \left\| X_{d_2k}-\mathbf{1}\bar{x}_{d_2k} \right\| ^2 \right] 
\\
&+3d_{2}^{2}\gamma ^2\mathbb{E} \left[ \left\| Y_{d_2k}-\mathbf{1}\bar{y}_{d_2k} \right\| ^2 \right] +9d_{2}^{2}\gamma ^2n\sigma ^2
\\
&+96nd_{2}^{4}\gamma ^4L^3\mathbb{E} \left[ f\left( \bar{x}_{d_2k} \right) -f\left( x^* \right) \right],
\end{aligned}
\end{equation}
wherein inequality (a) we have used Lemma \ref{Lem_inner_tra_err}. Then, using Lemma \ref{Lem_tech_1} and letting $\gamma \leqslant \frac{1}{8d_2L}$, we complete the proof.
\end{proof}

\subsection{Missing Proofs in the Main-text}

\subsubsection{Proof of Lemma \ref{Lem_opt_gap}}
\label{Proof_of_lem_opt_gap}
Using the x-update \eqref{Eq_x_update}, we have
\begin{equation}\label{Proof_lem1_1}
\begin{aligned}
&\mathbb{E} \left[ \left\| \bar{x}_{d_2\left( k+1 \right)}-x^* \right\| ^2 \right] 
\\
&=\mathbb{E} \left[ \left\| \bar{x}_{d_2k}-x^* \right\| ^2 \right]
+\gamma ^2\underset{=:S_2}{\underbrace{\mathbb{E} \left[ \left\| \sum_{j=0}^{d_2-1}{\frac{\mathbf{1}^T}{n}Y_{d_2k+j}} \right\| ^2 \right] }}
\\
&-2\gamma \underset{=:S_3}{\underbrace{\mathbb{E} \left[ \left< \bar{x}_{d_2k}-x^*,\sum_{j=0}^{d_2-1}{\frac{\mathbf{1}^T}{n}\nabla F_{d_2k+j}} \right> \right] }}.
\end{aligned}
\end{equation}

Next, we bound the terms $S_2$ and $S_3$ in the above equation respectively.
For $S_2$, noticing that $\mathbf{1}^TY_{d_2k+j}=\mathbf{1}^T\nabla G_{d_2k+j}$,
we have
\begin{equation}
\begin{aligned}
&\mathbb{E} \left[ \left\| \sum_{j=0}^{d_2-1}{\frac{\mathbf{1}^T}{n}Y_{d_2k+j}} \right\| ^2 \right] 
\\
&\overset{\left( a \right)}{\leqslant}\frac{d_2\sigma ^2}{n}+\mathbb{E} \left[ \left\| \sum_{j=0}^{d_2-1}{\frac{\mathbf{1}^T}{n}\nabla F_{d_2k+j}} \right\| ^2 \right] 
\\
&\overset{\left( b \right)}{\leqslant}\frac{2d_2L^2}{n}\underset{=:S_4}{\underbrace{\sum_{j=0}^{d_2-1}{\mathbb{E} \left[ \left\| X_{d_2k+j}-\mathbf{1}\bar{x}_{d_2k} \right\| ^2 \right]}}}
\\
&+2d_{2}^{2}\mathbb{E} \left[ \left\| \nabla f\left( \bar{x}_{d_2k} \right) \right\| ^2 \right] +\frac{d_2\sigma ^2}{n}
\\
&\leqslant \frac{2d_2L^2}{n}S_4+4d_{2}^{2}L\left( f\left( \bar{x}_{d_2k} \right) -f\left( x^* \right) \right) +\frac{d_2\sigma ^2}{n},
\end{aligned}
\end{equation}
wherein the inequality $(a)$ we have used Assumption \ref{Ass_bounded_var}, and $(b)$ used the smoothness of $f_i$.
For $S_3$, using the convexity and smoothness of $f_i$ in Assumption \ref{Ass_joint_smo_cov}, we have
\begin{equation}
\begin{aligned}
&\mathbb{E} \left[ \left< \bar{x}_{d_2k}-x^*,\sum_{j=0}^{d_2-1}{\frac{\mathbf{1}^T}{n}\nabla F_{d_2k+j}} \right> \right] 
\\
&=\sum_{j=0}^{d_2-1}{\frac{1}{n}\sum_{i=1}^n{\mathbb{E} \left[ \left< x_{i,d_2k+j}-x^*,\nabla f_i\left( x_{i,d_2k+j} \right) \right> \right]}}
\\
&-\sum_{j=0}^{d_2-1}{\frac{1}{n}\sum_{i=1}^n{\mathbb{E} \left[ \left< x_{i,d_2k+j}-\bar{x}_{d_2k},\nabla f_i\left( x_{i,d_2k+j} \right) \right> \right]}}
\\
&\geqslant d_2\mathbb{E} \left[ f\left( \bar{x}_{d_2k} \right) -f\left( x^* \right) \right]
\\
&+\frac{\mu}{2n}\sum_{j=0}^{d_2-1}{\mathbb{E} \left[ \left\| X_{d_2k+j}-\mathbf{1}x^* \right\| ^2 \right]}-\frac{L}{2n}S_4.
\end{aligned}
\end{equation}

Substituting $S_2$ and $S_3$ in \eqref{Proof_lem1_1}, and noticing that
\begin{equation}
\begin{aligned}
&-\sum_{j=0}^{d_2-1}{\mathbb{E} \left[ \left\| X_{d_2k+j}-\mathbf{1}x^* \right\| ^2 \right]}
\\
&=-\sum_{j=0}^{d_2-1}{\mathbb{E} \left[ \left\| X_{d_2k+j}-\mathbf{1}\bar{x}_{d_2k} \right\| ^2 \right]}-n\mathbb{E} \left[ \left\| \bar{x}_{d_2k}-x^* \right\| ^2 \right] 
\\
&-2\sum_{j=0}^{d_2-1}{\sum_{i=1}^n{\left( \mathbb{E} \left[ \left< x_{i,d_2k+j}-\bar{x}_{d_2k},\bar{x}_{d_2k}-x^* \right> \right] \right)}}
\\
&\leqslant S_4-\frac{nd_2}{2}\mathbb{E} \left[ \left\| \bar{x}_{d_2k}-x^* \right\| ^2 \right] ,
\end{aligned}
\end{equation}
we get
\begin{equation}
\begin{aligned}
&\mathbb{E} \left[ \left\| \bar{x}_{d_2\left( k+1 \right)}-x^* \right\| ^2 \right] 
\leqslant \left( 1-\frac{d_2\mu \gamma}{2} \right) \mathbb{E} \left[ \left\| \bar{x}_{d_2k}-x^* \right\| ^2 \right] 
\\
&+\frac{2\gamma L\left( 1+d_2\gamma L \right)}{n}S_4+\frac{d_2\gamma ^2\sigma ^2}{n}
\\
&-\left( 2d_2\gamma -4d_{2}^{2}\gamma ^2L \right) \mathbb{E} \left[ f\left( \bar{x}_{d_2k} \right) -f\left( x^* \right) \right].
\end{aligned}
\end{equation}
We note that it is only necessary to bound $S_4$ with Lemma \ref{Lem_tech_client_diver} if $d_2\geqslant2$. Then, letting the stepsize satisfy $\gamma \leqslant \frac{1}{8d_2L}$, we complete the proof.

\subsubsection{Proof of Lemma \ref{Lem_cons_err}}
\label{Proof_of_lem_con_err}
Using the x-update \eqref{Eq_x_update}, we have
\begin{equation}
\begin{aligned}
&\mathbb{E} \left[ \left\| X_{d_2\left( k+1 \right)}-\mathbf{1}\bar{x}_{d_2\left( k+1 \right)} \right\| ^2 \right] 
\\
&\leqslant \frac{1+\rho _{W}^{d_1}}{2}\mathbb{E} \left[ \left\| X_{d_2k}-\mathbf{1}\bar{x}_{d_2k} \right\| ^2 \right] 
\\
&+\frac{\gamma ^2d_2\left( 1+\rho _{W}^{d_1} \right) \rho _{W}^{d_1}}{1-\rho _{W}^{d_1}}\sum_{t=0}^{d_2-1}{\mathbb{E} \left[ \left\| Y_{d_2k+t}-\mathbf{1}\bar{y}_{d_2k+t} \right\| ^2 \right]},
\end{aligned}
\end{equation}
where we have used Young's inequality. By Lemma \ref{Lem_inner_tra_err} with $d_2\geqslant2$, letting the stepsize satisfy \eqref{Eq_step_cons_err}, we complete the proof.

\subsubsection{Proof of Lemma \ref{Lem_tra_err}}
\label{Proof_of_lem_tra_err}
Applying the recursion of FlexGT \eqref{Eq_y_update}, by Assumption \ref{Ass_bounded_var}, we have
\begin{equation}
\begin{aligned}
&\mathbb{E} \left[ \left\| Y_{d_2\left( k+1 \right)}-\mathbf{1}\bar{y}_{d_2\left( k+1 \right)} \right\| ^2 \right] 
\\
&\overset{\left( a \right)}{\leqslant}\frac{1+\rho _{W}^{d_1}}{2}\mathbb{E} \left[ \left\| Y_{d_2k}-\mathbf{1}\bar{y}_{d_2k} \right\| ^2 \right] +2n\rho _{W}^{d_1}\sigma ^2
\\
&+\frac{\left( 1+\rho _{W}^{d_1} \right) \rho _{W}^{d_1}}{1-\rho _{W}^{d_1}}\mathbb{E} \left[ \left\| \nabla F_{d_2\left( k+1 \right)}-\nabla F_{d_2k} \right\| ^2 \right] 
\\
&+2\rho _{W}^{d_1}\mathbb{E} \left[ \left< Y_{d_2k}-\mathbf{1}\bar{y}_{d_2k}, \nabla F_{d_2k}-\nabla G_{d_2k} \right> \right] 
\\
&+2\rho _{W}^{d_1}\mathbb{E} \left[ \left< \nabla F_{d_2k}-\nabla G_{d_2k}, \nabla F_{d_2\left( k+1 \right)} \right> \right] 
\\
&\overset{\left( b \right)}{\leqslant}\frac{1+\rho _{W}^{d_1}}{2}\mathbb{E} \left[ \left\| Y_{d_2k}-\mathbf{1}\bar{y}_{d_2k} \right\| ^2 \right] 
\\
&+\frac{6\rho _{W}^{d_1}L^2}{1-\rho _{W}^{d_1}}\mathbb{E} \left[ \left\| X_{d_2k}-\mathbf{1}\bar{x}_{d_2k} \right\| ^2 \right] 
\\
&+\frac{6\rho _{W}^{d_1}L^2}{1-\rho _{W}^{d_1}}\mathbb{E} \left[ \left\| X_{d_2\left( k+1 \right)}-\mathbf{1}\bar{x}_{d_2k} \right\| ^2 \right] +5n\rho _{W}^{d_1}\sigma ^2,
\end{aligned}
\end{equation}
wherein the inequality $(a)$ we have used Young's inequality and $\mathbb{E} \left[ \nabla G_{d_2k} \right] =\nabla F_{d_2k}$, and $(b)$ we have used Assumption \ref{Ass_joint_smo_cov} and the fact $\mathbb{E} \left[ \left< Y_{d_2k}-\mathbf{1}\bar{y}_{d_2k},\nabla F_{d_2k}-\nabla G_{d_2k} \right> \right] \leqslant n\sigma ^2$.
Then, adapting Lemma \ref{Lem_tech_client_diver} to the case $t=d_2\geqslant2$,
and letting the stepsize satisfy \eqref{Eq_step_tracking_err}, we complete the proof.
\balance
\bibliographystyle{ieeetr}
\bibliography{reference}
\end{document}